\documentclass[12pt]{interact}

\usepackage{graphicx}
\usepackage{amsmath,amssymb,amsfonts}
\usepackage{amsthm,enumerate,mathrsfs}
\usepackage{subcaption}
\usepackage{tikz}
\usetikzlibrary{shapes.geometric, arrows.meta}
\usetikzlibrary{arrows.meta,positioning,backgrounds}
\usetikzlibrary{fit,spy,calc,shapes.misc,matrix}
\usepackage{pgfplots}
\pgfplotsset{compat=1.17}
\usepackage{enumitem}
\usepackage{float}
\usepackage{booktabs}
\usepackage{url}
\usepackage{listings}
\usepackage{xcolor}
\usepackage[title]{appendix}
\usepackage[numbers,sort&compress]{natbib}
\usepackage[doublespacing]{setspace}
\bibpunct[, ]{[}{]}{,}{n}{,}{,}% Citation support using natbib.sty
\makeatletter% @ becomes a letter
\def\NAT@def@citea{\def\@citea{\NAT@separator}}% Suppress spaces between citations using natbib.sty
\makeatother% @ becomes a symbol again

\theoremstyle{plain}% Theorem-like structures provided by amsthm.sty
\newtheorem{theorem}{Theorem}[section]
\newtheorem{lemma}[theorem]{Lemma}
\newtheorem{corollary}[theorem]{Corollary}
\newtheorem{proposition}[theorem]{Proposition}

\theoremstyle{definition}
\newtheorem{definition}[theorem]{Definition}
\newtheorem{example}[theorem]{Example}

\theoremstyle{remark}
\newtheorem{remark}{Remark}

\lstdefinestyle{GAPstyle}{
    language=GAP,
    basicstyle=\ttfamily\footnotesize,
    keywordstyle=\color{blue}\bfseries,
    commentstyle=\color{green!60!black},
    stringstyle=\color{red},
    numbers=left,
    numberstyle=\tiny\color{gray},
    stepnumber=1,
    numbersep=5pt,
    breaklines=true,
    frame=lines,
    backgroundcolor=\color{gray!5},
    captionpos=b,
    escapeinside={(*@}{@*)}
}
\begin{document}

%\articletype{ARTICLE TEMPLATE}% Specify the article type or omit as appropriate

\title{Generalized Latin Square Graphs of Semigroups: A Counting Framework for Regularity and Spectra}

\author{
\name{M. R. Sorouhesh\textsuperscript{a}\thanks{Email: babaxor@iau.ac.ir} , M. Golriz\textsuperscript{a}\thanks{Email: Mgolriz@iau.ac.ir} and B. Panbehkar\textsuperscript{b}\thanks{Email: Bo.Panbehkar@iau.ac.ir}}
\affil{\textsuperscript{a}Department of Mathematics, ST. C., Islamic Azad University, Tehran, Iran; \textsuperscript{b}Department of Mathematics, Dez. C., Islamic Azad University, Dezful, Iran}
}

\maketitle

\begin{abstract}
We introduce the \emph{Generalized Latin Square Graph} $\Gamma(S)$ of a finite semigroup $S$. 
Since we record global factorization multiplicities and local alternative counts, we define three counting invariants $N_S,N_R,N_C$. 
This gives that we have a simple degree formula
\[
\deg(v)=2n-3+Q(v),\qquad Q(v)=N_S(s_k)-2N_R(v)-2N_C(v).
\]
We show that $\Gamma(S)$ is regular exactly when $Q$ is constant. 
We apply the framework to cancellative semigroups, bands, Brandt semigroups and null semigroups. 
For null semigroups, since we identify $\Gamma(S)\cong K_n\times K_n$, we compute the spectrum and energy. 
A concise computational appendix lists the \texttt{GAP} driver and representative outputs.
\end{abstract}


\section{Introduction}\label{sec:intro}
Graphs derived from \textbf{algebraic} structures serve as a bridge between algebra and \textbf{combinatorics}. Since the theory for groups is well-established \cite{Cameron2022}, semigroups present challenges due to the failure of cancellation law. This gives that structural differences influence graph-theoretic properties such as \textbf{regularity} and spectral distribution.

A \emph{Latin square} of order \(n\) is an \(n\times n\) array with entries from a set \(S\) in which each symbol appears exactly once in every row and column. The classical Latin square graph has vertex set consisting of triples \((r,c,s)\); two distinct vertices are adjacent precisely when they share exactly one coordinate \cite{Denes1974}. For a finite group, the multiplication table is a Latin square, and the associated graph is known to be regular.

In this paper, we introduce the \emph{Generalized Latin Square Graph} $\Gamma(S)$, which is a construction that encodes ordered factorization within a semigroup. Unlike the standard Cayley graphs $\mathrm{Cay}(S,A)$, which depend on the choice of a generating set $A$ \cite{Wang2013}, the graph $\Gamma(S)$ is intrinsic to the semigroup structure. The energy of a graph $G$ is defined as the sum of the absolute values of the eigenvalues of its adjacency matrix. Since we record ordered factorization \textbf{multiplicities} as vertex data, we establish a direct link between multiplication table statistics and graph invariants.

\subsection{Related Work}
Graphs defined on semigroups and groups have been studied from several perspectives. \textit{Power graphs} and their variants were introduced in works (see, e.g., \cite{Kelarev2003,Chakrabarty2009,Chattopadhyay2021}). Moreover, \textit{Cayley graphs} provide a generator-dependent viewpoint \cite{Wang2013}. In addition, commuting graphs and other constructions have been investigated for various semigroup families \cite{Paulista2025,Steinberg2016}.

Spectral investigations of power-type graphs are relevant to our discussion. The energy of a graph $G$ is defined as the sum of the absolute values of the eigenvalues of its adjacency matrix. For example, spectra of power graphs for finite groups are presented in \cite{Mehranian2017}. While those works focus on power graphs, our graph $\Gamma(S)$ is built from the full multiplication table. Since our graph encodes ordered factorization multiplicities, the two approaches are different.

Our primary tool is a set of three counting invariants ($N_S, N_R, N_C$). This framework yields an explicit degree formula 
\[
\deg(v)=2n-3+Q(v),
\]
and gives a criterion for regularity. We apply these results to cancellative semigroups, bands, Brandt semigroups, and null semigroups. For null semigroups, since we identify the graph structure as a tensor product, we compute its exact spectrum and energy.

The paper is organized as follows. Section~\ref{sec:preliminaries} fixes notation. Section~\ref{sec:generalized} defines the graph $\Gamma(S)$. Section~\ref{sec:degree} derives the degree formula. Section~\ref{sec:special_classes} applies the framework to special classes, and Section~\ref{sec:null} handles the spectral analysis for null semigroups.

\begin{figure}[H]
\centering
\begin{tikzpicture}[
    scale=0.95,
    node/.style={circle, draw=black!70, fill=gray!20, minimum size=2mm, inner sep=1pt}
]
\def\R{4.0}

\foreach \i in {1,2,3} {
    \foreach \j in {1,2,3} {
        \pgfmathtruncatemacro{\labeli}{(\i-1)*3+\j}
        \node[node] (v\i\j) at (\j*1.5, -\i*1.5) {};
        \node[above=0.1cm of v\i\j, font=\small] {($s_\i, s_\j, 0$)};
    }
}
% Draw Tensor Product edges (Row and Column disjoint connections)
\foreach \i in {1,2,3} {
    \foreach \j in {1,2,3} {
        \foreach \p in {1,2,3} {
            \foreach \q in {1,2,3} {
                \ifnum \i=\p \else
                    \ifnum \j=\q \else
                        \draw[gray!40, thin] (v\i\j) -- (v\p\q);
                    \fi
                \fi
            }
        }
    }
}
\node[font=\small] at (3, -5) {$\Gamma(S) \cong K_3 \times K_3$};
\end{tikzpicture}
\caption{Visualization of $\Gamma(S)$ for a null semigroup of order $n=3$.}
\label{fig:null_tensor}
\end{figure}
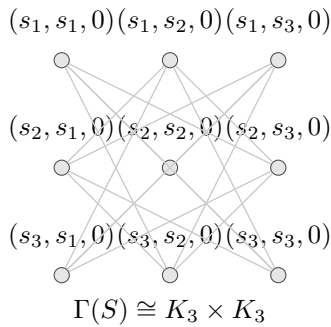
\section{Preliminaries}\label{sec:preliminaries}
All semigroups in this paper are finite. 
Let $S=\{s_1,\dots,s_n\}$ be a semigroup of order $n$. We recall definitions of specific semigroup classes studied in this paper. For general semigroup theory, the reader can consult \cite{Howie1995,Clifford1961,Steinberg2016}.

A semigroup $S$ is left (resp. right) cancellative if $sx=sy$ (resp. $xs=ys$) implies $x=y$. If both hold, $S$ is cancellative. A semigroup $S$ is a \textit{band} if $x^2=x$ for all $x \in S$. A band is \textit{rectangular} if it is isomorphic to the direct product $L \times R$ of a left zero semigroup $L$ and a right zero semigroup $R$. Let $G$ be a group and $I$ be a non empty set. The \textit{Brandt semigroup} $B_0(G,I)$ consists of elements $(i,g,j)$ for $i,j \in I, g \in G$ and a zero element $0$. Multiplication is defined as: $(i,g,j)(k,h,l) = (i, gh, l)$ if $j=k$, and $0$ otherwise. A semigroup $S$ is a null semigroup with zero $0$ if $xy=0$ for all $x,y \in S$. A semigroup $S$ is said to be an inverse semigroup if for every element $x \in S$, there exists a unique element $y \in S$ such that $x=xyx$ and $y=yxy$.

\paragraph{Notation and conventions.}
\begin{itemize}
  \item $\Gamma(S)$: the \emph{Generalized Latin Square Graph} of $S$ (Definition~\ref{def:glsg}).
  \item $N_S(s_k)$: number of ordered pairs $(x,y)\in S^2$ with $xy=s_k$.
  \item For a vertex $v=(s_i,s_j,s_k)$ set
  \[
    N_R(v):=|\{t\in S\setminus\{s_j\}: s_i t = s_k\}|,
    \qquad
    N_C(v):=|\{t\in S\setminus\{s_i\}: t s_j = s_k\}|.
  \]
  \item $K_n$: complete graph on $n$ vertices.
  \item We use the tensor (Kronecker) product of graphs. Concretely,
  \[
    V(G\times H)=V(G)\times V(H),\qquad
    (u,v)\sim(u',v')\iff u\sim_G u'\ \text{and}\ v\sim_H v'.
  \]
\end{itemize}

\section{Generalized Latin Square Graphs}\label{sec:generalized}
\begin{definition}\label{def:glsg}
Let $S$ be a finite semigroup of order $n$. 
The \emph{Generalized Latin Square Graph} $\Gamma(S)$ is the simple graph with vertex set
\[
V(\Gamma(S))=\{(s_i,s_j,s_k)\in S^3\mid s_i s_j = s_k\}.
\]
Two distinct vertices $(s_i,s_j,s_k)$ and $(s_p,s_q,s_r)$ are adjacent iff they agree in exactly one coordinate.
\end{definition}

\begin{remark}
Since vertices sharing two coordinates represent the same product, we exclude them from adjacency. This restriction preserves the generalized Latin square structure, as allowing such edges would create dependencies stronger than those in classical Latin square graphs. Hence, the graph remains simple and geometrically consistent with the multiplication table structure.
\end{remark}

\begin{example}\label{ex:simple}
Let $S_1=\{s_0,s_1\}$ with $s_1$ an identity and $s_0$ a left-zero. 
Then
$V(\Gamma(S_1))=\{(s_0,s_0,s_0),(s_0,s_1,s_0),(s_1,s_0,s_0),(s_1,s_1,s_1)\}$. 
So, we have that the first vertex is isolated and the other three form a $K_3$.
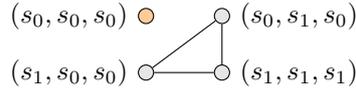
\begin{figure}[h]
\centering
\begin{tikzpicture}[scale=0.5, every node/.style={circle, draw, fill=gray!20, minimum size=2mm, inner sep=0pt}, label distance=3pt]
\node (a) at (0,0)    [fill=orange!40, label=left:{\small$(s_0,s_0,s_0)$}] {};
\node (b) at (2,0)    [label=right:{\small$(s_0,s_1,s_0)$}] {};
\node (c) at (0,-1.5) [label=left:{\small$(s_1,s_0,s_0)$}] {};
\node (d) at (2,-1.5) [label=right:{\small$(s_1,s_1,s_1)$}] {};
\draw (d)--(b);\draw (d)--(c);\draw (b)--(c);
\end{tikzpicture}
\caption{The generalized Latin square graph $\Gamma(S_1)$.}\label{fig:glsg_s1}
\end{figure}
\end{example}

\section{Degree Formula and Regularity}\label{sec:degree}
For $v=(s_i,s_j,s_k)\in V(\Gamma(S))$ define the neighbour classes:
\begin{align*}
N_1(v) &= \{(s_i,t,s_i t)\in V(\Gamma(S))\mid t\in S,\ t\neq s_j,\ s_i t\neq s_k\},\\
N_2(v) &= \{(t,s_j,t s_j)\in V(\Gamma(S))\mid t\in S,\ t\neq s_i,\ t s_j\neq s_k\},\\
N_3(v) &= \{(t,u,s_k)\in V(\Gamma(S))\mid t,u\in S,\ t\neq s_i,\ u\neq s_j,\ t u = s_k\}.
\end{align*}

\begin{lemma}\label{lem:partition}
The sets $N_1(v),N_2(v),N_3(v)$ are pairwise disjoint and exhaust $N(v)$. Moreover
\[
|N_1(v)|=(n-1)-N_R(v),\qquad |N_2(v)|=(n-1)-N_C(v),
\]
\[
|N_3(v)|=N_S(s_k)-1-N_R(v)-N_C(v).
\]
\end{lemma}
\begin{proof}
Since adjacency forces agreement in exactly one coordinate, we have three cases. 
Disjointness is obvious. 
Moreover, counting follows by fixing the free coordinate and excluding those choices that produce $s_k$. This completes the proof.
\end{proof}

\begin{theorem}\label{thm:degree}
For $v=(s_i,s_j,s_k)\in V(\Gamma(S))$ we have
\[
\deg(v)=2n-3+Q(v),\qquad Q(v)=N_S(s_k)-2N_R(v)-2N_C(v).
\]
Moreover $\Gamma(S)$ is regular iff $Q$ is constant on $V(\Gamma(S))$.
\end{theorem}
\begin{proof}
We sum the three cardinalities in Lemma~\ref{lem:partition}. So, we get the formula. Hence the proof is complete.
\end{proof}

\subsection{Bounds and boundary characterizations}
We collect trivial bounds and give short examples.

\begin{lemma}\label{lem:bounds_combined}
Let $v=(s_i,s_j,s_k)\in V(\Gamma(S))$ and $n=|S|$. Then
\begin{enumerate}
  \item $0\le N_R(v),N_C(v)\le n-1$. Equality $N_R(v)=n-1$ (resp.\ $N_C(v)=n-1$) holds iff $s_i t=s_k$ for every $t\in S\setminus\{s_j\}$ (resp.\ $t s_j=s_k$ for every $t\in S\setminus\{s_i\}$).
  \item $N_S(s_k)\ge 1$. Moreover $N_S(s_k)=1$ iff the only ordered pair producing $s_k$ is $(s_i,s_j)$.
\end{enumerate}
\end{lemma}
\begin{proof}
(1) By definition, $N_R(v)$ counts elements in $S \setminus \{s_j\}$, which has size $n-1$. Thus $0 \le N_R(v) \le n-1$. Since $N_R(v)=n-1$ holds iff this subset equals the entire set, so $s_i t = s_k$ for every $t \neq s_j$. The argument for $N_C(v)$ is symmetric.
(2) Since $v=(s_i,s_j,s_k)$ is a vertex, $s_i s_j = s_k$. Thus at least one ordered pair produces $s_k$, so $N_S(s_k) \ge 1$. If $N_S(s_k)=1$, then $(s_i,s_j)$ must be the unique pair. This completes the proof.
\end{proof}

\begin{remark}
If $N_R(v)=n-1$ then $|N_1(v)|=0$. 
If $N_S(s_k)=1$ then $|N_3(v)|=0$.
\end{remark}

\section{Regularity in Special Classes}\label{sec:special_classes}

\subsection{Cancellative semigroups and groups}
In a cancellative semigroup, the equations $s_i t = s_k$ and $t s_j = s_k$ admit unique solutions 
\cite{Howie1995,Clifford1961}.
Since $(s_i,s_j,s_k)$ is a vertex, the unique solutions must be $t=s_j$ and $t=s_i$, respectively.

\begin{corollary}\label{cor:group_regularity}
If $S$ is cancellative then $N_R(v)=N_C(v)=0$ for every $v$. Hence
 $\deg(v)=2n-3+N_S(s_k)$. If $S$ is a group then $N_S(s_k)=n$ for all $s_k$, so $\Gamma(S)$ is $(3n-3)$-regular.
\end{corollary}
\begin{proof}
By the cancellation property, the equations have unique solutions. Since $s_i s_j = s_k$, so the solutions are uniquely $t=s_j$ and $t=s_i$. Therefore, $N_R(v)=0$ and $N_C(v)=0$. So, we substitute these into Theorem~\ref{thm:degree} and get the result. Hence the proof is complete.
\end{proof}
\begin{proposition}[Classical Regularity \cite{Denes1974}]
The Latin square graph associated with a finite group of order \(n\) is a regular graph of degree \(3n-3\).
\end{proposition}
\begin{proof}
This follows directly from Corollary~\ref{cor:group_regularity}.
\end{proof}
\subsection{Bands}
\begin{proposition}
If $S$ is a rectangular band $L\times R$ then $\Gamma(S)$ is regular.
\end{proposition}

\begin{proof}
Let $v=(s_i,s_j,s_k)$ be a vertex. Write $s_i=(\ell_i,r_i)$, $s_j=(\ell_j,r_j)$ and $s_k=(\ell_k,r_k)$. 
Since multiplication in $L\times R$ yields $(\ell_i,r_i)(\ell_j,r_j)=(\ell_i,r_j)$,  the product $s_i s_j=s_k$ implies 
\[
\ell_k=\ell_i,\qquad r_k=r_j.
\]
Any ordered pair $(s,t)=(\ell,r)(\ell',r')$ equals $s_k$ when $\ell=\ell_k$ and $r'=r_k$. Thus, 
\[
N_S(s_k)=|R|\cdot|L|,
\]
which depends only on $|L|$ and $R$.
For the vertex $v$, the count $N_R(v)$ is the number of choices of $t=(\ell_t,r_t)$ with $r_t=r_j$ (excluding $t=s_j$). Hence 
\[
N_R(v)=|L|-1.
\]
By symmetry, 
\[
N_C(v)=|R|-1.
\]
These counts depend only on $|L|$ and $|R|$. So, $Q(v)$ is independent of $v$. By Theorem~\ref{thm:degree}, $\Gamma(S)$ is regular. Hence the proof is complete.
\end{proof}

\begin{remark}
The result is invariant under swapping $L$ and $R$. Since the counts
\[
N_S(s_k)=|L|\cdot|R|,\qquad N_R(v)=|L|-1,\qquad N_C(v)=|R|-1
\]
depend only on $|L|,|R|$, the quantity $Q=N_S-2N_R-2N_C$ is symmetric. Hence exchanging $L$ and $R$ does not affect regularity.
\end{remark}
\begin{lemma}[Spectrum of Rectangular Bands]\label{lem:rectangular_spectrum}
Let $S = L \times R$ be a rectangular band. Since the multiplication is given by $(l,r)(l',r') = (l,r')$, the vertex $v=(s_i,s_j,s_k)$ belongs to a block determined by its $L$-class. Consequently, the adjacency matrix of $\Gamma(S)$ is block-diagonal, and its spectrum is the union of the spectra of these blocks.
\end{lemma}
\begin{proof}
The vertices of $\Gamma(S)$ are partitioned by their $L$-coordinates. Two vertices are adjacent only if they agree on the $L$-coordinate. Hence, there are no edges between distinct blocks. Therefore, the adjacency matrix is block-diagonal, implying the claimed spectral decomposition. This completes the proof.
\end{proof}
\subsection{Inverse semigroups with unique solvability}\label{sec:inverse}
\begin{proposition}\label{prop:unique_solvability}
Let $S$ be a finite inverse semigroup. If for a vertex $v=(s_i,s_j,s_k)$ the equations $s_i x=s_k$ and $y s_j=s_k$ admit unique solutions, then $N_R(v)=N_C(v)=0$ and $\deg(v)=2n-3+N_S(s_k)$.
\end{proposition}

\begin{proof}
Since the equation $s_i x=s_k$ has a unique solution ($x=s_j$), no $t \neq s_j$ satisfies $s_i t=s_k$. Thus, $N_R(v)=0$. By symmetry, $N_C(v)=0$. So, substituting these into Theorem~\ref{thm:degree} gives the result. Hence the proof is complete.
\end{proof}

\begin{remark}
The hypothesis of unique solvability is essential. Since if only one equation admits a unique solution, then only the corresponding alternative count vanishes. So, the degree formula must be applied with the nonzero count retained.
\end{remark}

\subsection{Brandt semigroups}\label{sec:brandt}
Let $S=B_0(G,I)$ be the Brandt semigroup with structure group $G$ of order $m$ and index set $I$ of size $n$. Nonzero elements are triples $(i,g,j)$ and there is a zero $0$. Then $|S|=1+mn^2$ (see \cite{Howie1995,Clifford1961}).

\begin{lemma}[Sufficient impossibility condition]\label{lem:impossible}
Let $S$ be finite of order $s$. Suppose $z_0, z_1\in S$ and set $\Delta=N_S(z_0)-N_S(z_1)$. 
If $\Delta>4(s-1)$ then $Q$ cannot be constant on $V(\Gamma(S))$. In particular $\Gamma(S)$ is not regular.
\end{lemma}
\begin{proof}
For any vertex $v$ we have $0\le 2N_R(v)+2N_C(v)\le 4(s-1)$. Since if $Q$ were constant then the difference $\Delta$ would equal the difference of two such terms, $\Delta\le 4(s-1)$. This is a contradiction. This completes the proof.
\end{proof}

\begin{proposition}\label{prop:brandt}
If $n>1$ then $\Gamma(B_0(G,I))$ is not regular.
\end{proposition}
\begin{proof}
Let $z=(i,g,j)$ be a nonzero target. Ordered pairs $(x,y)$ with $xy=z$ are exactly
\[
x=(i,g_1,k),\qquad y=(k,g_2,j),
\]
with $k\in I$ and $g_1g_2=g$. For each $k$ there are $m$ choices for $g_1$. Thus
 $N_S(z)=mn$.

Total ordered pairs in $S\times S$ is $(1+mn^2)^2$. Since pairs producing a nonzero product are those nonzero pairs with matching middle index, we have $m^2 n^3$ such pairs. Hence
\[
N_S(0)=(1+mn^2)^2-m^2 n^3.
\]
Set $\Delta=N_S(0)-N_S(z)$. A short expansion yields
\[
\Delta=1+2mn^2-mn+m^2 n^3(n-1).
\]
We distinguish cases based on the order $m$ of the structure group $G$.
\begin{itemize}
    \item \textbf{Case $m \ge 2$:} Since the term $m^2 n^3(n-1)$ is non-negative and dominant for $n > 1$, we verify $\Delta > 4mn^2$:
    \[
    1+2mn^2-mn+m^2 n^3(n-1) > 4mn^2 \iff m^2 n^3(n-1) > 2mn^2 + mn - 1.
    \]
    This inequality holds for all $n>1$. Thus $\Delta > 4(|S|-1)$, and by Lemma~\ref{lem:impossible}, $Q$ cannot be constant.
    \item \textbf{Case $m = 1$:} Here $S$ is the Brandt semigroup over the trivial group. 
    Since for any nonzero target $z$, $N_S(z)=n$, equations $s_i t = z$ and $t s_j = z$ have unique solutions. This implies that $N_R(v)=N_C(v)=0$. Hence $Q(v)=n$. 
    Conversely, for a vertex $w=(x,y,0)$ with $x,y \neq 0$, we have $N_S(0) = (1+n^2)^2 - n^3$. For $n>1$, this value grows like $n^4$, which is larger than $n$. So, $Q(w) = N_S(0) - 2n(n-1) \neq Q(v)$. Therefore $Q$ is not constant.
\end{itemize}
In all cases with $n>1$, $\Gamma(S)$ is not regular. Hence the proof is complete.
\end{proof}

\begin{remark}[Numeric illustrations]
Two small cases for $m \ge 2$:
\begin{itemize}
  \item $m=2,n=2$. Then $|S|=9$, $N_S(\text{nonzero})=4$, $N_S(0)=49$, so $\Delta=45$ and $4(|S|-1)=32$.
  \item $m=2,n=3$. Then $|S|=19$, $N_S(\text{nonzero})=6$, $N_S(0)=253$, so $\Delta=247$ and $4(|S|-1)=72$.
\end{itemize}
In both cases $\Delta>4(|S|-1)$, so regularity is impossible.
\end{remark}
\iffalse
\begin{table}[h]
\centering
\caption{Brandt examples: numeric diagnostics}
\label{tab:brandt_numeric}
\begin{tabular}{@{\extracolsep{5pt}}lrrrrr}
\toprule
Parameters & $s$ & $N_S(\text{nonzero})$ & $N_S(0)$ & $\Delta$ & $4(s-1)$\\
\midrule
$m=2,\ n=2$ & $1+2\cdot 2^2=9$ & $4$ & $49$ & $45$ & $32$\\
$m=2,\ n=3$ & $1+2\cdot 3^2=19$ & $6$ & $253$ & $247$ & $72$\\
\bottomrule
\end{tabular}
\end{table}\fi

\subsection{Constant-image semigroups}
A semigroup $S$ is called a \emph{constant-image semigroup} if there exists an element $c \in S$ such that $S^2=\{c\}$; ie; $xy=c$ for all $x,y \in S$. In this case, the vertex set of $\Gamma(S)$ consists entirely of triples $(s_i, s_j, c)$. Since for any such vertex $(s_i, s_j, c)$ the equations $s_i t = c$ and $t s_j = c$ are satisfied for every $t \in S$, so the counts of alternative factors are maximal.

\begin{proposition}
If $S$ satisfies $S^2=\{c\}$ then $\Gamma(S)$ is $(n-1)^2$-regular.
\end{proposition}
\begin{proof}
Since $xy=c$ for all $x,y \in S$, the total number of factorizations is $N_S(c) = n^2$.
For any vertex $v=(s_i, s_j, c)$, the condition $s_i t = c$ holds for all $t \in S$. Therefore, the number of alternative elements $t \in S \setminus \{s_j\}$ is exactly $n-1$, giving $N_R(v) = n-1$. Symmetrically, $N_C(v) = n-1$.
So, we substitute these values into the degree formula from Theorem~\ref{thm:degree}:
\begin{align*}
\deg(v) &= 2n-3 + [N_S(c) - 2N_R(v) - 2N_C(v)],\\
&= 2n-3 + [n^2 - 2(n-1) - 2(n-1)],\\
&= 2n-3 + (n^2 - 4n + 4) = (n-1)^2.
\end{align*}
Thus, the degree is constant for all vertices. Hence the proof is complete.
\end{proof}

\section{Null Semigroups: Connectivity and Spectrum}\label{sec:null}
Let $S$ be a null semigroup of order $n$ with zero $0$ and $s_i s_j=0$ for all $i,j$. Then
\[
V(\Gamma(S))=\{(s_i,s_j,0):1\le i,j\le n\}.
\]
\noindent Unlike the classical case for groups, where row-neighbors are adjacent due to sharing a single coordinate, the shared zero product here causes row-neighbors to share two coordinates. This distinction disconnects the grid structure and leads to the tensor product identification.

Recall that vertices are adjacent iff they agree in exactly one coordinate. 
Since any two vertices agree on the third coordinate $0$, so they are adjacent only if they \textbf{disagree} on the first two coordinates. 
Thus, the map $(s_i,s_j,0)\mapsto (i,j)$ is an isomorphism
\[
\Gamma(S)\cong K_n\times K_n.
\]
Eigenvalues of $K_n$ are $n-1$ (mult.~$1$) and $-1$ (mult.~$n-1$). Since we use Kronecker product, eigenvalues are products of eigenvalues of factors. Hence, the spectrum of $\Gamma(S)$ is
\[
(n-1)^2\ (1),\quad -(n-1)\ (2(n-1)),\quad 1\ ((n-1)^2).
\]
Therefore, the graph energy is
\[
E(\Gamma(S)) = (n-1)^2 + 2(n-1)(n-1) + (n-1)^2 = 4(n-1)^2.
\]

\subsection{Connectivity}

Since connectivity of $K_n\times K_n$ depends on $n$, $\Gamma(S)$ is connected for $n \ge 3$, while for $n=2$ it consists of two disjoint components.
\begin{figure}[H]
\centering
\begin{tikzpicture}[
    scale=0.95,
    node/.style={circle, draw=black!70, fill=gray!20, minimum size=2mm, inner sep=1pt},
    labelstyle/.style={font=\small}
]
% Grid for n=2
\foreach \i in {1,2} {
    \foreach \j in {1,2} {
        \node[node] (v\i\j) at (\j*2.0, -\i*2.0) {};
        \node[labelstyle, above=0.12cm of v\i\j] {($s_{\i}, s_{\j}, 0$)};
    }
}
\draw[thick] (v11) -- (v22);
\draw[thick] (v12) -- (v21);

\node[font=\small] at (2, -4.6) {$\Gamma(S) \cong K_2 \times K_2$ (two disjoint edges)};
\end{tikzpicture}
\caption{Visualization of $\Gamma(S)$ for a null semigroup of order $n=2$.}
\label{fig:null_tensor_n2}
\end{figure}
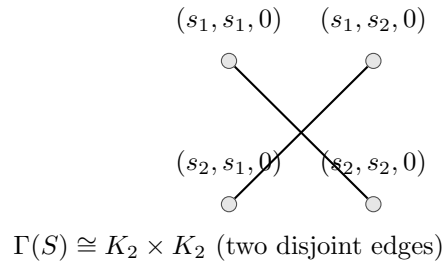

\section{Computational Verification and Examples}\label{sec:experiments}
We checked counting formulas and representative examples using \texttt{GAP} (version 4.14.0) for small orders (up to $n=6$). The checks served as tests for Theorem~\ref{thm:degree} and illustrated typical behaviours.

\paragraph{Verification strategy.}
For a multiplication table $M$, a naive degree calculation by enumerating all adjacencies requires $O(n^4)$ operations. Since computing $N_R$ and $N_C$ repeatedly is costly, we adopted an optimized $O(n^2)$ approach. The script \texttt{VerifyGammaFast.g} implements this algorithm and verifies the degree formula for any input semigroup.

\paragraph{Worked example (null semigroup, $n=3$).}
Let $S=\{s_1,s_2,s_3\}$ be a null semigroup with product $s_3$. The output of \texttt{VerifyGammaFast} confirms that $\Gamma(S)$ is regular. For a sample vertex $v=(s_1,s_2,3)$, the script computes
\[
N_S(3)=9,\qquad N_R(v)=N_C(v)=2,\qquad \deg(v)=4,
\]
in agreement with Theorem~\ref{thm:degree}. A console excerpt demonstrating the full degree check is available in Appendix~\ref{app:gap}.

\paragraph{Using the verification script.}
To verify the degree formula for any semigroup, one can simply read the multiplication table (e.g., generated by standard \texttt{GAP} constructors like \texttt{NullSemigroup} or \texttt{BrandtSemigroup}) into the variable \texttt{M}, load the verification script, and run the command \texttt{VerifyGammaFast(M)}. The script outputs a detailed breakdown of degrees and confirms whether $\Gamma(S)$ is regular.

\paragraph{Representative GAP commands.}
The following snippet demonstrates the workflow for the null semigroup example presented above:
\begin{verbatim}
M := [ [3,3,3], [3,3,3], [3,3,3] ];

# Run verification
Read("VerifyGammaFast.g");
res := VerifyGammaFast(M);
Print(res.Degrees);
\end{verbatim}

\section{Computational Analysis}\label{sec:prevalence}
In order to estimate the prevalence of regular generalized Latin square graphs, an exhaustive enumeration was performed for all semigroups of order $n \le 6$. The computations were conducted using the \texttt{smallsemi} \cite{Distler2024} and \texttt{semigroups} \cite{MitchellSemigroups} packages within the \texttt{GAP} system \cite{GAP2024}. The enumeration provides a sample size of over 17,000 semigroups. The data is consistent with known sequences of semigroup counts (cf. OEIS A001423 \cite{OEIS2024}).

\paragraph{Algorithm.} To determine whether $\Gamma(S)$ is regular, we implemented an optimized algorithm in \texttt{GAP}. The algorithm precomputes the counting invariants $N_S, N_R, N_C$ for each vertex and checks if the quantity $Q = N_S - 2N_R - 2N_C$ is constant. This script was used to verify the regularity condition on the entire dataset of small semigroups.

\begin{table}[ht]
\centering
\caption{Prevalence of regular Generalized Latin Square Graphs (GLSGs) for semigroups of small order.}
\label{tab:prevalence}
\begin{tabular}{cccc}
\toprule
\textbf{Order $n$} & \textbf{Total Semigroups} & \textbf{Regular GLSGs} & \textbf{Percentage} \\ 
\midrule
1 & 1 & 1 & 100.0 \\ 
2 & 4 & 3 & 75.0 \\ 
3 & 18 & 3 & 16.7 \\ 
4 & 126 & 8 & 6.3 \\ 
5 & 1,160 & 3 & 0.26 \\ 
6 & 15,973 & 12 & 0.08 \\ 
\midrule
\textbf{Total} & \textbf{17,282} & \textbf{30} & \textbf{0.17} \\ 
\bottomrule
\end{tabular}
\end{table}

\paragraph{Algebraic Obstructions to Regularity.}
The significant decrease in frequency observed in Table~\ref{tab:prevalence} reflects a fundamental algebraic obstruction. Since the regularity of $\Gamma(S)$ is equivalent to the condition that $Q(i, j)$ remains constant, so we must analyze this condition. Each of the three counting functions involved depends on distinct structural features of the multiplication table. For a generic semigroup, as $n$ increases, the distribution of factorization multiplicities broadens. The requirement that these distinct variations compensate one another imposes a system of $n^2$ simultaneous constraints on the multiplication table. Consequently, the probability that a randomly selected semigroup satisfies these constraints approaches zero as $n$ increases.

\section{Conclusion and Future Work}\label{sec:conclusion}
We presented a canonical graph $\Gamma(S)$ for finite semigroups and a compact counting framework. 
The degree formula links algebraic factorization counts to graph regularity. 
Applications include bands, Brandt semigroups, and null semigroups. 
This paper establishes a bridge between factorization complexity and graph regularity, and shows that observed irregularities in $\Gamma(S)$ can be largely explained by variations in the local factorization counts $N_R$ and $N_C$. 

\begin{remark}[Block structure and spectrum]
If $S$ decomposes as a semilattice of subsemigroups (e.g., \textbf{Clifford semigroups}, which are precisely semilattices of disjoint groups), then $\Gamma(S)$ inherits a block structure. In particular, for certain inverse semigroups arising as semilattices of groups, this yields a practical route to compute or bound eigenvalues.
\end{remark}

\textbf{Future work :} Systematic spectral study for inverse semigroups and tighter bounds on $Q$ via structural properties are interesting topics for research.

\appendix
\section{Computational materials and representative outputs} 
\subsection{Optimized Verification Script}\label{app:gap}
The following \texttt{GAP} script, \texttt{VerifyGammaFast.g}, implements the $O(n^2)$ degree verification algorithm described in Section~\ref{sec:experiments}. It defines a function that accepts any multiplication table $M$ (represented as a list of lists) and returns the degree list.

\begin{lstlisting}[style=GAPstyle, caption={\texttt{VerifyGammaFast.g}: Optimized verification driver.}]
# VerifyGammaFast.g - optimized verification driver (GAP, 1-based)
# Usage:
#   Read("VerifyGammaFast.g");
#   M := [[1,2,3], ...];  # User defines their table
#   res := VerifyGammaFast(M);

LoadPackage("semigroups");
VerifyGammaFast := function(tbl)
    local n, i, j, k,
          N_S, NR_Table, NC_Table,
          deg_list, Q_list,
          NR_total, NC_total, NR, NC, Q, deg;

    # n: number of elements (assume tbl is n x n and entries in 1..n)
    n := Length(tbl);

    # Initialize N_S and NR/NC tables (1-based indexing)
    N_S := List([1..n], _ -> 0);
    NR_Table := List([1..n], _ -> List([1..n], _ -> 0)); # NR_Table[i][k]
    NC_Table := List([1..n], _ -> List([1..n], _ -> 0)); # NC_Table[j][k]

    # Single pass: fill N_S, NR_Table and NC_Table in O(n^2)
    for i in [1..n] do
        for j in [1..n] do
            k := tbl[i][j];            # k in 1..n
            N_S[k] := N_S[k] + 1;
            NR_Table[i][k] := NR_Table[i][k] + 1;
            NC_Table[j][k] := NC_Table[j][k] + 1;
        od;
    od;

    # Compute degrees using precomputed tables
    deg_list := [];
    Q_list := [];

    for i in [1..n] do
        for j in [1..n] do
            k := tbl[i][j];  # third coordinate (1..n)

            # totals including current (i,j) entry
            NR_total := NR_Table[i][k];
            NC_total := NC_Table[j][k];

            # alternatives exclude current factor (exclude s_j from NR, s_i from NC)
            NR := NR_total - 1;   # number of t != s_j with i*t = k
            NC := NC_total - 1;   # number of t != s_i with t*j = k

            # compute Q and degree (note N_S[k] already 1-based)
            Q := N_S[k] - 2*NR - 2*NC;
            deg := 2*n - 3 + Q;

            Add(Q_list, Q);
            Add(deg_list, deg);
        od;
    od;

    # Print summary
    if ForAll(deg_list, d -> d = deg_list[1]) then
        Print("\nResult: Gamma(S) is REGULAR. Degree = ", deg_list[1], "\n");
    else
        Print("\nResult: Gamma(S) is NOT REGULAR.\n");
        Print("Degree set: ", Set(deg_list), "\n");
    fi;

    return rec(N_S := N_S, NR_Table := NR_Table, NC_Table := NC_Table,
               Q_list := Q_list, Degrees := deg_list);
end;;

# ==========================================
# Example Usage (User can change M to any multiplication table)
# Example: Null semigroup of order 3 (all products map to element 3)
# Note: Index 3 corresponds to the zero element.
M := [[3,3,3], [3,3,3], [3,3,3]]; 
res := VerifyGammaFast(M);

# Printing results for inspection
Print("N_S table: ", res.N_S);
Print("Degrees: ", res.Degrees);
\end{lstlisting}
\section{Computational materials and representative outputs}\label{app:gap_analysis}

\subsection{Regular GLSG Analysis Script}
To ensure the reproducibility of the prevalence data in Section~\ref{sec:prevalence}, we provide the \texttt{GAP} driver \texttt{RunGLSGAnalysis}. The script iterates over all semigroups up to a specified order $n$ and checks the regularity of their generalized Latin square graphs using the precomputed counting invariants. A concise description of the algorithm is given below:

\begin{lstlisting}[style=GAPstyle, caption={\texttt{RunGLSGAnalysis.g}: Optimized verification driver.}]
# ---------------------------------------------------------
# GAP: Optimized Regularity Analysis of Small Semigroups
# ---------------------------------------------------------

LoadPackage("semigroups");
LoadPackage("smallsemi");

# Main Driver Function
RunGLSGAnalysis := function(maxOrder)
    local n, sems, tot, reg, sem, perc, t0, t1;

    for n in [1..maxOrder] do
        t0 := Runtime();
        
        # Enumerate all semigroups of order n
        sems := AllSmallSemigroups(n);
        tot := Size(sems);
        reg := 0;

        # Check regularity for each semigroup
        for sem in sems do
            if IsRegularGLSG(sem) then
                reg := reg + 1;
            fi;
        od;

        perc := Float(reg) / tot * 100.0;
        t1 := Runtime();

        # Print row for Table 3
        Print("Order ", n, ": ", reg, "/", tot, " regular (",
              String(perc), "%)  [", String(t1-t0)/1000.0, " s]\n");
    od;
end;

# ---------------------------------------------------------
# Helper: IsRegularGLSG (Optimized Implementation)
# Checks if Q is constant for a given semigroup S
# ---------------------------------------------------------
IsRegularGLSG := function(S)
    local elts, n, mult, i, j, k, p, q, NS, NR, NC, Qvals, NS_table;

    elts := Elements(S);
    n := Size(elts);
    mult := List(elts, x -> List(elts, y -> Position(elts, x * y)));

    # Precompute N_S for all elements
    NS_table := List([1..n], k -> 0);
    for p in [1..n] do
        for q in [1..n] do
            k := mult[p][q];
            NS_table[k] := NS_table[k] + 1;
        od;
    od;

    Qvals := [];

    for i in [1..n] do
        for j in [1..n] do
            k := mult[i][j];

            # Count N_R and N_C
            NR := Size(Filtered([1..n], q -> q <> j and mult[i][q] = k));
            NC := Size(Filtered([1..n], p -> p <> i and mult[p][j] = k));

            # Compute Q
            Q := NS_table[k] - 2 * (NR + NC);
            Add(Qvals, Q);
        od;
    od;

    # Regular if all Q values are equal
    return Size(Set(Qvals)) = 1;
end;

# ---------------------------------------------------------
# Execute Analysis
# Note: Computing order 6 may take significant time and RAM.
# ---------------------------------------------------------
RunGLSGAnalysis(6);
\end{lstlisting}

\end{document}